\documentclass[12pt]{article}
\usepackage{a4}
\usepackage{amsthm}
\usepackage{amsfonts}
\usepackage{amssymb}
\usepackage{amsmath}
\usepackage{cite}
\usepackage{epsfig}
\usepackage[hidelinks]{hyperref}

\newtheorem{theorem}{Theorem}
\newtheorem{lemma}[theorem]{Lemma}
\newtheorem{corollary}[theorem]{Corollary}
\newtheorem{proposition}[theorem]{Proposition}

\newtheorem{problem}{Problem}

\newcommand\NN{{\mathbb N}}
\newcommand\ZZ{{\mathbb Z}}
\newcommand\RR{{\mathbb R}}
\newcommand{\dd}{\,\mathrm{d}}

\begin{document}
\title{Sidorenko property and forcing in regular tournaments\thanks{The first and third authors were supported by the Alexander von Humboldt Foundation in the framework of the Alexander von Humboldt Professorship of the first author endowed by the Federal Ministry of Education and Research.}}
\author{Daniel Kr\'al'\thanks{Institute of Mathematics, Leipzig University, Augustusplatz 10, 04109 Leipzig, and Max Planck Institute for Mathematics in the Sciences, Inselstra{\ss}e 22, 04103 Leipzig, Germany. E-mail: {\tt daniel.kral@uni-leipzig.de}.}\and
        Matja\v{z} Krnc\thanks{Faculty of Mathematics, Natural Sciences and Information Technologies, and Andrej Maru\v{s}i\v{c} Institute, University of Primorska, Glagolja\v ska 8, SI-6000 Koper, Slovenia. E-mail: \texttt{matjaz.krnc@upr.si}. This author was supported by Slovenian Research and Innovation Agency (P1-0383 and N1-0370).}\and
        Filip Ku\v{c}er\'ak\thanks{Institute of Mathematics, Leipzig University, Augustusplatz 10, 04109 Leipzig. E-mail: {\tt filip.kucerak@mis.mpg.de}.}\and
	Bernard Lidick\'y\thanks{Department of Mathematics, Iowa State University, 411 Morrill Road, Ames, IA, 50011, USA. E-mail: \texttt{lidicky@iastate.edu}. Research of this author is supported in part by NSF grant FRG DMS-2152490, Simons Collaboration grant and a Scott Hanna professorship.}\and
        Jan Volec\thanks{Department of Theoretical Computer Science, Faculty of Information Technology, Czech Technical University in Prague, Th\'akurova 9, Prague, 160 00, Czech Republic. E-mail: \texttt{jan@ucw.cz}. This author was supported by the grant 23-06815M of the Grant Agency of the Czech Republic.}
}
\date{}
\maketitle

\begin{abstract}
We give a complete characterization of tournaments $H$ that have the Sidorenko property with respect to nearly regular tournaments,
i.e., the homomorphism density of $H$ among all nearly regular tournaments is minimized by a random tournament.
Corollaries of our result
are a positive answer to the question of Noel, Ranganathan and Simbaqueba
whether there exist infinitely many non-transitive tournaments that
are quasirandom forcing for nearly regular tournaments, and
a negative answer to their question
whether almost every tournament is quasirandom forcing for nearly regular tournaments.
\end{abstract}

\section{Introduction}
\label{sec:intro}

The work presented in this paper
is motivated by problems concerning quasirandomness of tournaments (orientations of complete graphs).
Informally speaking, a combinatorial structure is said to be \emph{quasirandom}
if it has properties that a random structure would have asymptotically almost surely.
The study of~\emph{quasirandom graphs} can be traced back
to the nowadays classical works of R\"odl~\cite{Rod86}, Thomason~\cite{Tho87,Tho87b} and
Chung, Graham and Wilson~\cite{ChuGW89} from 1980s.
The notion of quasirandom graphs is particularly robust as
there are seemingly different characterizations of quasirandom graphs, such as
through homomorphism counts, the distribution of edges, the cut sizes, algebraic properties, etc., and
so it has found applications in many different settings.
There is a long series of results concerning quasirandomness of other kinds of combinatorial structures,
for example
groups~\cite{Gow08},
hypergraphs~\cite{ChuG90,ChuG91s,Gow06,Gow07,HavT89,KohRS02,NagRS06,RodS04},
permutations~\cite{ChaKNPSV20,Coo04,KraLN24,KraP13},
tournaments~\cite{BucLSS21,ChuG91,CorR17,GrzIKK23,HanKKMPSV23,NoeRS25},
subsets of integers~\cite{ChuG92}, etc.

We are interested in \emph{quasirandom forcing} substructures.
We illustrate this property on quasirandom graphs, likely the most studied notion of quasirandom structures.
A graph $H$ is quasirandom forcing if the following holds for every sequence $(G_n)_{n\in\NN}$ of graphs:
the sequence $(G_n)_{n\in\NN}$ is quasirandom if and only if
the limit of the homomorphism density of $H$ in $(G_n)_{n\in\NN}$ is equal to the expected homomorphism density of $H$ in a random graph.
In other words, a graph is quasirandom if and only if the homomorphism density of $H$ is close to its expected density,
i.e., any non-randomness necessarily results in the deviation from the expected homomorphism density of $H$.
Examples of quasirandom forcing graphs include even cycles and complete bipartite graphs with each part of size at least two.

The concept of quasirandom forcing is intimately related to the notion of Sidorenko graphs.
A graph $H$ has the \emph{Sidorenko property}
if the homomorphism density of $H$ is asymptotically minimized by a random graph.
One of the most intriguing questions in extremal combinatorics is
a conjecture of Sidorenko~\cite{Sid93} and of Erd\H{o}s and Simonovits~\cite{ErdS83} that
asserts that every bipartite graph has the Sidorenko property;
we refer particularly to~\cite{BlaR65, Sid89, Sid91, ConFS10, ConL17, ConKLL18, ConL21}
for classes of bipartite graphs proven to have the Sidorenko property.
It is easy to show that every quasirandom forcing graph must have the Sidorenko property, and
the Forcing Conjecture of Conlon, Fox and Sudakov~\cite{ConFS10},
a well-known generalization of the above mentioned conjecture based on a question of Skokan and Thoma~\cite{SkoT04},
is equivalent to the statement that a graph $H$ is quasirandom forcing if and only if $H$ is bipartite and has at least one cycle.

\subsection{Quasirandomness in tournaments}

Our work is motivated by quasirandom forcing in the setting of tournaments.
It is well-known that transitive tournaments have the Sidorenko property,
i.e., their homomorphism density is asymptotically minimized by a random tournament, and
every transitive tournament with with at least four vertices
is quasirandom forcing, see~\cite{CorR17} and \cite[Exercise 10.44]{Lov93}.
Coregliano, Parente and Sato~\cite{CorPS19} identified a non-transitive quasirandom forcing tournament on five vertices,
which is depicted in Figure~\ref{fig:T5}.
In fact, this 5-vertex tournament has the \emph{anti-Sidorenko property},
i.e., its homomorphism density is asymptotically maximized by a random tournament.
Buci\'{c}, Long, Shapira and Sudakov~\cite{BucLSS21} observed that
there are no additional quasirandom forcing tournaments with seven or more vertices, and
the remaining tournaments on at most six vertices were analyzed by Hancock et al.~\cite{HanKKMPSV23}.
Hence, the transitive tournaments with at least four vertices and
the $5$-vertex tournament depicted in Figure~\ref{fig:T5} are the only quasirandom forcing tournaments,
which is in contrast with the setting of graphs
where very rich families of quasirandom forcing graphs are known~\cite{ConFS10, ConL17, ConKLL18, ConL21}.

\begin{figure}
\begin{center}
\epsfbox{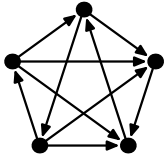}
\end{center}
\caption{The unique quasirandom forcing tournament that is not transitive.}
\label{fig:T5}
\end{figure}

Noel, Ranganathan and Simbaqueba~\cite{NoeRS25} considered quasirandom forcing in the setting of nearly regular tournaments.
A sequence $(T_n)_{n\in\NN}$ of tournaments is \emph{nearly regular}
if for every $\varepsilon>0$,
the proportion of vertices in $T_n$ whose out-degree differ from $|V(T_n)|/2$ by more than $\varepsilon|V(T_n)|$ tends to zero.
Clearly, if a tournament $H$ is quasirandom forcing for all sequences of tournaments,
then $H$ is quasirandom forcing for nearly regular sequences.
In the other direction, if $H$ is quasirandom forcing for nearly regular sequences of tournaments and
$H'$ is any tournament forcing near regularity (for example $H'$ can be chosen
to be the cyclically oriented triangle or the $3$-vertex transitive tournament),
then $\{H,H'\}$ is a quasirandom forcing family for all sequences of tournaments.

Noel, Ranganathan and Simbaqueba~\cite{NoeRS25} characterized all tournaments with at most five vertices that
are quasirandom forcing for nearly regular sequences of tournaments.
In particular, they identified three such additional $4$-vertex tournaments and five $5$-vertex tournaments.
Furthermore, they posed the following three open problems,
out of which we answer two of them and we resolve the remaining one up to finitely many cases.

\begin{problem}[{Noel, Ranganathan and Simbaqueba~\cite[Problem 6.1]{NoeRS25}}]
\label{prob:1}
Characterize tournaments that are quasirandom forcing for nearly regular sequences of tournaments.
\end{problem}

\begin{problem}[{Noel, Ranganathan and Simbaqueba~\cite[Question 6.2]{NoeRS25}}]
\label{prob:2}
Are there infinitely many non-transitive tournaments that are quasirandom forcing for nearly regular sequences of tournaments?
\end{problem}

\begin{problem}[{Noel, Ranganathan and Simbaqueba~\cite[Question 6.3]{NoeRS25}}]
\label{prob:3}
Is almost every tournament quasirandom forcing for nearly regular sequences of tournaments?
\end{problem}

\subsection{Our results}

The main result of this paper determines the complete list of tournaments that have the Sidorenko property for nearly regular sequences of tournaments.
In particular, it provides infinitely many such non-transitive tournaments, which answers the question given in Problem~\ref{prob:2} in the affirmative.
\begin{theorem}\label{thm:main}
Let $H$ be a non-transitive tournament.
The asymptotic density of $H$ in every nearly regular sequence of tournaments is at least $2^{-\binom{|V(H)|}{2}}$ 
if and only if
$H$ is isomorphic to $T[a,b,c]$ for some $a,b,c \in \NN$,
where $T[a,b,c]$ denotes the blow-up of the cyclically oriented triangle with the parts inducing transitive tournaments of orders $a$, $b$ and $c$, respectively.
Moreover, $T[a,b,c]$ is quasirandom forcing for nearly regular sequences of tournaments unless $a=b=c=1$.
\end{theorem}

A moment of thought reveals that every quasirandom forcing tournament must have either the Sidorenko property or the anti-Sidorenko property.
An analogous statement holds also when only nearly regular sequences of tournaments are considered (sc. Proposition~\ref{prop:SidAntiSid}).
Regarding tournaments with the anti-Sidorenko property in this regular setting, we show in Section~\ref{sec:antiSidorenko} that they can have at most nine vertices.
In particular, any tournament not captured by Theorem~\ref{thm:main} that is quasirandom forcing for nearly regular sequences of tournaments has at most nine vertices,
which answers the question in Problem~\ref{prob:3} in the negative, and solves Problem~\ref{prob:1} up to finitely many cases.

\begin{theorem}
\label{thm:max}
For every tournament $H$ with at least $10$ vertices there exists a nearly regular sequence of tournaments such that the asymptotic density of $H$ in the sequence 
is larger than $2^{-\binom{|V(H)|}{2}}$.
\end{theorem}

\subsection{Sketch of the proof of Theorem~\ref{thm:main}}

Before proceeding with presenting our arguments, we would like to briefly highlight the main steps and ideas.
We treat the problem in the language of combinatorial limits, which we introduce in Section~\ref{sec:prelim};
we refer the reader to this section for notions used in this paragraph that has not yet been defined.
In Section~\ref{sec:constr},
we show that if $H$ is neither a transitive tournament nor a blow-up of the cyclically oriented triangle as in the characterization,
then there are regular tournaments that are $H$-free;
as discussed in the beginning of Section~\ref{sec:constr},
this statement also follows from the results presented in~\cite{BolH90}.

The core part of our proof of Theorem~\ref{thm:main} is presented in Section~\ref{sec:Sidorenko}.
An easy argument (formally given in Lemma~\ref{lm:trans}) yields that
it is enough to show that the blow-ups of the cyclically oriented triangle with independent parts, which are denoted by $C[a,b,c]$ (see Figure~\ref{fig:Cabc}),
are quasirandom forcing for nearly regular tournaments.

Let us now sketch the proof for the case $a=b$, which we believe to transparently capture the main idea.
A key observation is that $t(C[1,1,c],W)=t(B[c],W)/2$ for any regular tournamenton $W$
where $B[c]$ is the digraph obtained from $C[1,1,c]$
by removing the edge between the parts of size one (see Figure~\ref{fig:Bc}).
We next define three auxiliary functions (the formal definitions are given before Lemma~\ref{lm:Cabk}):
$N_{W,c}^{+}:[0,1]^c\to [0,1]$ that measures the size of the ``common out-neighborhood'' of a $c$-tuple of points,
$N_{W,c}^-:[0,1]^c\to [0,1]$ that measures the size of the ``common in-neighborhood'', and
$D_{W,c}:[0,1]^c\to [0,1]$ that measures the density of edges from the ``common out-neighborhood'' to the ``common in-neighborhood''.
We immediately obtain that
\begin{align*}
t(B[c],W) & =\int N_{W,c}^{-}(x)N_{W,c}^{+}(x)\dd x_{[c]}\qquad\mbox{and}\\
t(C[1,1,c],W) & =\int N_{W,c}^{-}(x)D_{W,c}(x)N_{W,c}^{+}(x)\dd x_{[c]}.
\end{align*}
On the other hand, the Sidorenko property of complete bipartite graphs oriented from one part to another yields that
\begin{equation}
t(C[a,a,c],W)\ge \int N_{W,c}^{-}(x)^aD_{W,c}(x)^{a^2}N_{W,c}^{+}(x)^{a}\dd x_{[c]}.
\label{eq:tCaac}
\end{equation}
Jensen's and H\"older's inequalities now imply that
\[t(C[1,1,c],W)\le t(B[c],W)^{\frac{a-1}{a}} t(C[a,a,c],W)^{\frac{1}{a^2}}.\]
In general, it is hard to bound $t(B[c],W)$ in estimates as the one above (an analogous issue has prevented us
from extending the entropy proof presented in Section~\ref{sec:concl} to all values of $a$, $b$ and $c$), however,
as we pointed out, it holds that $t(C[1,1,c],W)=t(B[c],W)/2$ for regular tournamentons $W$.
Hence, we obtain that
\[2^{-a^2}\;t(B[c],W)^a\le t(C[a,a,c],W);\]
this estimate yields the result as
$B[c]$ can be shown to have the Sidorenko property and to be quasirandom forcing by standard arguments.
The actual proof of Theorem~\ref{thm:Cabc} consists of
a generalization of the Sidorenko type inequality \eqref{eq:tCaac} and
careful applications of Jensen's Inequality and H\"older's Inequality.

\section{Preliminaries}
\label{sec:prelim}

In this section, we overview the notation used throughout the paper.
We first start with an overview of the general notation that we use and is less standard.
We write $[k]$ for the set $\{1,\ldots,k\}$.
We also use $\ZZ_k$ for the set $[k]$
when the additional algebraic structure given by the addition modulo $k$ is of importance.
In general, all integrals are over the space $[0,1]^k$ with Borel measure unless specified otherwise.
Finally, we write $x_A$ for $x\in\RR^A$, i.e., a vector whose coordinates are indexed by the elements of $A$.
Using the just introduced notation, $x_{[k]}$ is a vector $x\in\RR^k$ and
we will write $x_{[k]}$ instead of simply writing $x$
when we wish to emphasize the dimension of the vector $x$.
For example,
\[\int x_1x_2\dd x_{[2]}=\frac{1}{4}.\]

\begin{figure}
\begin{center}
\epsfbox{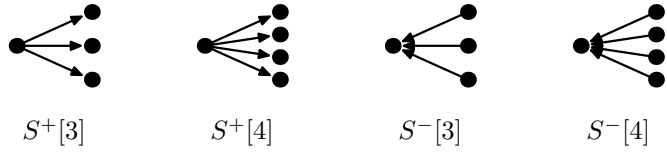}
\end{center}
\caption{The digraphs $S^+[3]$, $S^+[4]$, $S^-[3]$ and $S^-[4]$.}
\label{fig:Sk}
\end{figure}

\begin{figure}
\begin{center}
\epsfbox{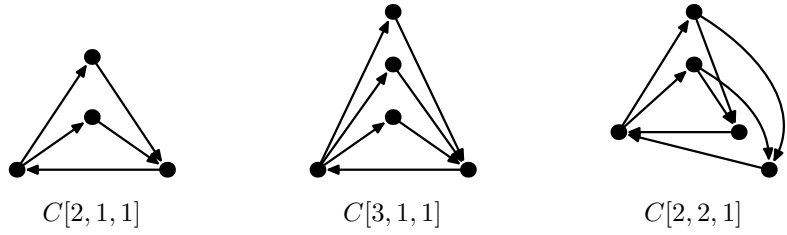}
\end{center}
\caption{The digraphs $C[2,1,1]$, $C[3,1,1]$ and $C[2,2,1]$,
         which are blow-ups of the cyclically oriented triangle with parts of the sizes given by the parameters.}
\label{fig:Cabc}
\end{figure}

\begin{figure}
\begin{center}
\epsfbox{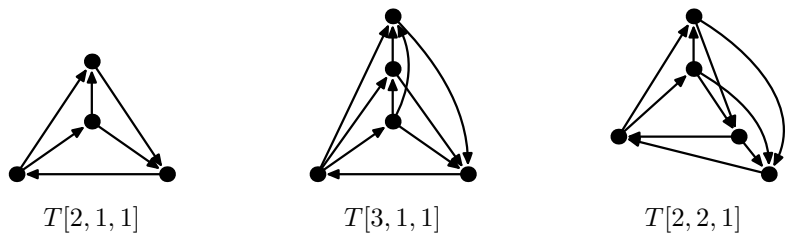}
\end{center}
\caption{The tournaments $T[2,1,1]$, $T[3,1,1]$ and $T[2,2,1]$.}
\label{fig:Tabc}
\end{figure}

We next introduce the notation related to digraphs and tournaments.
All digraphs considered in this paper are simple, i.e., without loops and parallel edges.
The vertices $u$ and $v$ of a digraph are \emph{twins}
if the out-neighbors of $u$ are exactly the out-neighbors of $v$,
the in-neighbors of $u$ are exactly the in-neighbors of $v$, and
there is no edge between $u$ and $v$.
We write $S^+[k]$ and $S^-[k]$, where $k\in\NN$, for the orientation of the $k$-leaf star such that the central vertex is the source and the sink, respectively; see Figure~\ref{fig:Sk} for illustrations.
Fix $a,b,c\in\NN$.
The digraph $C[a,b,c]$ is the blow-up of the cyclically oriented triangle with the parts of sizes $a$, $b$ and $c$ respectively,
i.e., $C[a,b,c]$ has three parts, one with $a$ vertices, one with $b$ vertices and one with $c$ vertices, and
it contains
all the edges between the $a$-vertex and the $b$-vertex parts directed to the $b$-vertex part,
all the edges between the $b$-vertex and the $c$-vertex parts directed to the $c$-vertex part, and
all the edges between the $a$-vertex and the $c$-vertex parts directed to the $a$-vertex part;
see Figure~\ref{fig:Cabc} for examples.
Note that all pairs of vertices contained in the same of the three parts of $C[a,b,c]$ are twins.
The tournament $T[a,b,c]$ is obtained from the digraph $C[a,b,c]$ by adding a transitive tournament
on the vertices of each of the three parts of the digraph $C[a,b,c]$;
see Figure~\ref{fig:Tabc} for examples.
Finally, recall the digraph $B[c]$ is obtained from $C[1,1,c]$
by removing the edge between the two vertices contained in the parts of size one; see Figure~\ref{fig:Bc}.
Note that $B[c]$ can be viewed as obtained from $c$ directed paths of length two
by identifying their first vertices to a single (source) vertex and their last vertices to a single (sink) vertex.

\begin{figure}
\begin{center}
\epsfbox{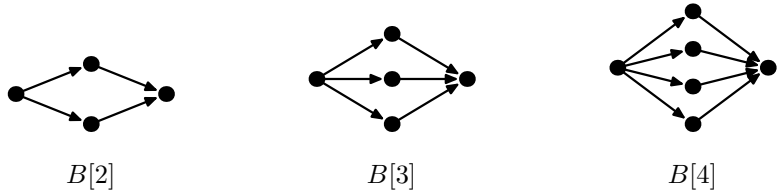}
\end{center}
\caption{The digraphs $B[2]$, $B[3]$ and $B[4]$.}
\label{fig:Bc}
\end{figure}

We next introduce the necessary concepts from the theory of combinatorial limits and
refer for a more thorough introduction to~\cite{HanKKMPSV23,NoeRS25},
where these tools were used in the context of quasirandom forcing tournaments.
A tournamenton is an analytic object that represents a convergent sequence of tournaments;
informally speaking, tournamentons can be thought of as adjacency matrices of large tournaments.
Formally, a \emph{tournamenton} is a measurable function $W:[0,1]^2\to [0,1]$
such that $W(x,y)+W(y,x)=1$ for all $(x,y)\in [0,1]^2$.
Quasirandom sequences of tournaments are represented by the constant tournamenton,
i.e., the tournament $W$ such that $W(x,y)=1/2$ for all $(x,y)\in [0,1]^2$.
Throughout the paper, we will use a shorthand notation $W\equiv 1/2$ to represent that
a tournamenton $W$ is equal to $1/2$ almost everywhere.
A tournamenton $W$ is \emph{regular} if
\[\int W(x,y)\dd y=\frac{1}{2}\]
for almost every $x\in [0,1]$;
loosely speaking, regular tournamentons represent large tournaments
where every vertex has asymptotically the same in-degree and out-degree.
Formally, regular tournamentons are limits of nearly regular convergent sequences of tournaments.

Let $H$ be a digraph. The \emph{homomorphism density} of $H$ in a tournamenton $W$,
denoted by $t(H,W)$ is defined as follows:
\begin{equation}
t(H,W)=\int\prod_{vw\in E(H)}W(x_v,x_w)\dd x_{V(H)}.
\label{eq:tHW}
\end{equation}
We write $t(H,1/2)$ for the homomorphism density of $H$ in the constant tournamenton.

We say that a digraph $H$ has the \emph{Sidorenko property} if $t(H,W) \ge 2^{-|E(H)|}$ for every tournamenton $W$, and
$H$ has the \emph{anti-Sidorenko property} if $t(H,W)\le 2^{-|E(H)|}$ for every tournamenton $W$.
Similarly,
a digraph $H$ has the \emph{Sidorenko property in regular tournamentons} if $t(H,W) \ge 2^{-|E(H)|}$ for every regular tournamenton $W$, and
$H$ has the \emph{anti-Sidorenko property in regular tournamentons} if $t(H,W)\le 2^{-|E(H)|}$ for every regular tournamenton $W$.
A digraph $H$ is \emph{quasirandom forcing} if a tournamenton $W$ satisfies that $t(H,W)=2^{-|E(H)|}$ if and only if $W\equiv 1/2$.
Similarly,  $H$ is \emph{quasirandom forcing in regular tournamentons}
if a regular tournamenton $W$ satisfies that $t(H,W)=2^{-|E(H)|}$ if and only if $W\equiv 1/2$.
As we mentioned, it is not hard to show that if $H$ is quasirandom forcing,
then $H$ has either the Sidorenko property or the anti-Sidorenko property, and
the same holds when restricted to the regular setting as given in Proposition~\ref{prop:SidAntiSid} below.

We next cast several classical results concerning tournaments in the language of combinatorial limits.
The first concerns the homomorphism density of transitive tournaments.

\begin{proposition}
\label{prop:trans}
Let $n\in\NN$ and let $T_n$ be the $n$-vertex transitive tournament.
For every tournamenton $W$, it holds that
\[t(T_n,W)\ge 2^{-\binom{n}{2}}.\]
Moreover, the equality holds if and only if
\begin{itemize}
\item $n\in\{1,2\}$ and $W$ is arbitrary,
\item $n=3$ and $W$ is regular, or
\item $n\ge 4$ and $W\equiv 1/2$.
\end{itemize}
\end{proposition}

Since there are only two $3$-vertex tournaments,
which are the transitive $3$-vertex tournament and the cyclically oriented triangle,
we derive from Proposition~\ref{prop:trans} the following.

\begin{corollary}
\label{cor:triangle}
Every tournamenton $W$ satisfies that $t(C[1,1,1],W)\le 1/8$ and
the equality holds if and only if $W$ is regular.
\end{corollary}

Let us also relate the quasirandom forcing property with the Sidorenko and the anti-Sidorenko properties for regular tournamentons.

\begin{proposition}
\label{prop:SidAntiSid}
If a digraph is quasirandom forcing in regular tournamentons, then it has either the Sidorenko property or the anti-Sidorenko property in regular tournamentons.
\end{proposition}

\begin{proof}
We prove the contrapositive statement:
given a digraph $H$ that has neither the Sidorenko nor the anti-Sidorenko property in regular tournamentons,
it holds that $H$ is not quasirandom forcing in regular tournamentons.
Fix any such an $H$, and let  $W_0$ and $W_1$ be two regular tournamentons that satisfy $t(H,W_0) < 2^{-|E(H)|} < t(H,W_1)$.

For any $\alpha \in (0,1)$, let $W_\alpha$ be the tournamenton defined as follows:
\[
W_\alpha(x,y)=\begin{cases}
       W_0\left(\frac{x}\alpha,\frac{y}\alpha\right) & \mbox{if $\left\{x,y\right\} \subseteq [0,\alpha)$,} \\
       W_1\left(\frac{x-\alpha}{1-\alpha},\frac{y-\alpha}{1-\alpha}\right) & \mbox{if $\left\{x,y\right\} \subseteq [\alpha,1)$, and} \\
       1/2 & \mbox{otherwise.}
       \end{cases}
\]
Informally speaking,
we take copies of $W_0$ and $W_1$ scaled by $\alpha$ and $(1-\alpha)$, respectively, and
orient all the edges between $W_0$ and $W_1$ randomly.

We next show that $W_{\alpha}$ is a regular tournamenton.
If $x\in [0,\alpha)$, it holds that
\[\int_{[0,1]}W(x,y)\dd y=
  \int_{[0,\alpha)}W_0\left(\frac{x}\alpha,\frac{y}\alpha\right)\dd y+\frac{1-\alpha}{2}=
  \frac{\alpha}{2}+\frac{1-\alpha}{2}=\frac{1}{2}.\]
Similarly, if $x\in [\alpha,1]$, it holds that
\[\int_{[0,1]}W(x,y)\dd y=
  \frac{\alpha}{2}+\int_{[\alpha,1]}W_1\left(\frac{x-\alpha}{1-\alpha},\frac{y-\alpha}{1-\alpha}\right)\dd y=
  \frac{\alpha}{2}+\frac{1-\alpha}{2}=\frac{1}{2}.\]
It follows that $W_{\alpha}$ is a regular tournamenton.

We now apply Intermediate Value Theorem to the function $F:[0,1]\to [0,1]$ defined as $F(z)=t(H,W_z)$;
note that the function $F$ is continuous on $[0,1]$.
It follows that there exists $\tau \in (0,1)$ such that $t(H,W_{\tau}) = 2^{-|E(H)|}$.
Since the tournamenton $W_\tau$ is not equal to $1/2$ everywhere,
we conclude that $H$ is not quasirandom forcing in regular tournamentons.
\end{proof}

We finish this section with a generalization of H\"older's Inequality, which will be used later.
Let $\Omega$ be a probability space with the probability measure $\mu$ and let $k$ be a positive integer. For any collection
of measurable functions $F_i:\Omega\to [0,1]$, $i\in [k]$, it holds that
\[\int_{\Omega} \prod_{i\in [k]}F_i(x)\dd \mu(x)\le\prod_{i\in [k]}\left(\int_{\Omega}F_i(x)^{p_i}\dd \mu(x)\right)^{1/p_i}\]
whenever $p_1,\ldots,p_k$ are non-negative reals such that $1/p_1+\cdots+1/p_k\le 1$.

\section{Constructions}
\label{sec:constr}

In this section, we present regular tournamentons witnessing that
a tournament $H$ does not have the Sidorenko property for nearly regular tournaments
unless $H$ is a transitive tournament or a tournament $T[a,b,c]$ for some $a,b,c\in\NN$.
This statement is also implied by~\cite[Theorem 4]{BolH90},
which gives a classification of tournaments $H$ such that
every sufficiently large regular tournament contains a copy of $H$:
tournaments with this property are referred to as omnipresent in~\cite{BolH90} and
a tournament is omnipresent if and only if it is a transitive tournament or
a tournament $T[a,b,c]$ for some $a,b,c\in\NN$.
For completeness of our presentation,
we decided to include a short argument leading to Theorem~\ref{thm:constr} in this paper.

A key step in this argument is the following lemma giving a structural characterization of tournaments $T[a,b,c]$.
To state the lemma, we introduce the following notation (see Figure~\ref{fig:special} for illustration):
$W_4$ is the $4$-vertex tournament with one vertex being a source and the remaining vertices forming a cyclically oriented triangle,
$L_4$ is the $4$-vertex tournament with one vertex being a sink and the remaining vertices forming a cyclically oriented triangle, and
$C_5$ is the $5$-vertex carousel tournament,
i.e., the vertex of $C_5$ can be viewed as $\ZZ_5$ and $uv$ is an edge if and only if $v-u\equiv1\mod 5$ or $v-u\equiv2\mod 5$.

\begin{figure}
\begin{center}
\epsfbox{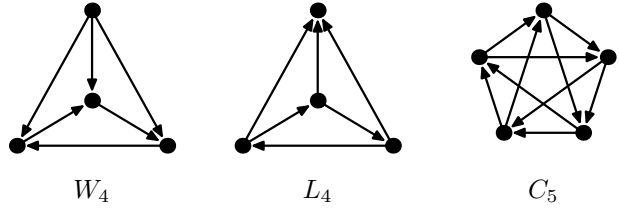}
\end{center}
\caption{The tournaments $W_4$, $L_4$ and $C_5$.}
\label{fig:special}
\end{figure}

\begin{lemma}
\label{lm:sub}
Every non-transitive tournament that contains neither of $W_4$, $L_4$ and $C_5$ is isomorphic to $T[a,b,c]$ for some $a,b,c\in\NN$.
\end{lemma}

\begin{proof}
Fix $H$ a non-transitive tournament that contains neither of the tournaments $L_4$, $W_4$ and $C_5$.
Since $H$ is not transitive, it contains a cyclically oriented triangle;
let $u_1u_2u_3$ be any cyclically oriented triangle of $H$.

We now partition the vertices $V$ into three sets $V_1$, $V_2$ and $V_3$ as follows.
First, for every $i\in [3]$, the vertex $u_i$ is included to the set $V_i$.
Let $v$ be a vertex of $V$ different from $u_1$, $u_2$ and $u_3$.
Since the tournament $H$ contains neither $W_4$ nor $L_4$,
the vertex $v$ has both an in-neighbor and an out-neighbor among the vertices $u_1$, $u_2$ and $u_3$.
If the vertex $v$ has exactly one in-neighbor among these three vertices, say $u_i$,
we add $v$ to the set $V_{i+1}$ (the index is taken modulo $3$).
If the vertex $v$ has exactly two in-neighbors among the vertices $u_1$, $u_2$ and $u_3$,
then $v$ has exactly one out-neighbor among them, say $u_j$,
we add $v$ to the set $V_{j-1}$ (again, the index is taken modulo $3$).
We conclude that the set $V_i$ is non-empty for every $i\in [3]$ ($V_i$ contains the vertex $u_i$), and
each vertex of the set $V_i$ has $u_{i-1}$ among its in-neighbors and $u_{i+1}$ among its out-neighbors.

\begin{figure}
\begin{center}
\epsfbox{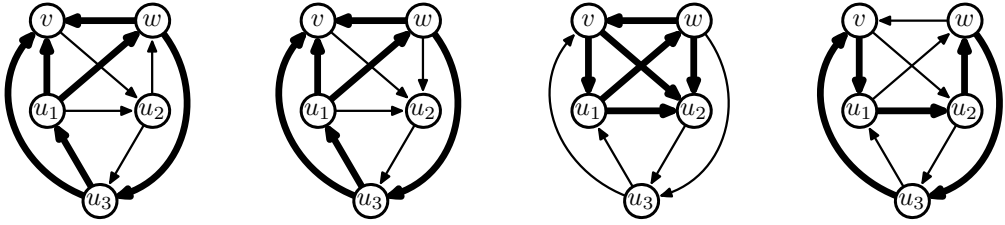}
\end{center}
\caption{The $5$-vertex tournament induced by vertices $u_i$, $u_{i+1}$, $u_{i+2}$, $v$ and $v'$ that
         depends on the direction of the edge between the vertices $u_i$ and $v$ and the edge between $u_{i+1}$ and $v'$.
	 The figure displays the case $i=1$ (other cases are symmetric).
	 The subgraphs isomorphic to $W_4$ or $L_4$ are depicted by bold edges in the first three cases and
	 the edges of one of the cycles of length five in $C_5$ are depicted by bold edges in the fourth case.}
\label{fig:subvv}
\end{figure}

Fix $i\in [3]$.
Consider any vertex $v$ contained in $V_i$ different from $u_i$ and
any vertex $w$ contained in $V_{i+1}$ different from $u_{i+1}$ (indices modulo $3$).
In case the edge between $v$ and $w$ was directed from $w$ to $v$,
the tournament $H$ would contain one of the tournaments $L_4$, $W_4$ and $C_5$ (the four cases that
depend on the direction of the edges $u_iv$ and $u_{i+1}w$ are drawn in Figure~\ref{fig:subvv}).
Hence, the edge between $v$ and $w$ is directed from $v$ to $w$.
Since the choice of $v$ and $w$ was arbitrary,
we conclude that all the edges of $H$ between $V_i$ and $V_{i+1}$ are directed from $V_i$ to $V_{i+1}$ (indices modulo $3$).

Finally, note that for every $i\in V_i$, every triple of vertices of $V_i$ induces a transitive tournament
(otherwise, the triple and the vertex $u_{i+1}$ would form the tournament $W_4$).
Hence, the tournaments induced by each of the sets $V_1$, $V_2$ and $V_3$ are transitive,
and $H$ is isomorphic to the tournament $T[|V_1|,|V_2|,|V_3|]$.
\end{proof}

We are now ready to prove the main theorem of this section.

\begin{theorem}
\label{thm:constr}
Let $H$ be a tournament that is neither a transitive tournament nor a tournament $T[a,b,c]$ for some $a,b,c\in\NN$.
Then, there exists a regular tournamenton $W$ such that $t(H,W)=0$.
\end{theorem}

\begin{proof}
By Lemma~\ref{lm:sub}, it is enough to establish the theorem for $H$ being $W_4$, $L_4$ and $C_5$.
Rather than exhibiting a specific regular tournamenton,
we will construct a sequence $(T_n)_{n\in\NN}$ of regular tournaments such that $t(H,T_n)=0$ and
the number of vertices of $T_n$ tends to infinity.
The sought tournamenton $W$ will be a limit tournamenton of a convergent subsequence of $(T_n)_{n\in\NN}$;
we remark that the sequences that we construct are actually convergent, however, we do not need this stronger claim.

Let $T_n$ be the $(2n+1)$-vertex carousel tournament,
i.e., the vertices of $T_n$ are $\ZZ_{2n+1}$ and
there is an edge directed from $x$ to $y$ if and only if $y-x\in\{1,\ldots,n\}$ (modulo $2n+1$).
Note that the out-neighbors of any vertex of $T_n$ induce a transitive tournament and
likewise the in-neighbors of any vertex induce a transitive tournament.
It follows that $t(W_4,T_n)=0$ and $t(L_4,T_n)=0$,
which establishes the statement when $H$ is $W_4$ or $L_4$.

It remains to consider the case $H=C_5$.
Let $T_n$ be the $n$-th iterated blow-up of the cyclically oriented triangle,
i.e., the vertices of $T_n$ are $\ZZ_3^n$ and
there is an edge directed from $x$ to $y$ if and only if $y_i-x_i=1$ (modulo $3$) for the smallest index $i$ such that $x_i\not=y_i$.
We claim that $t(C_5,T_n)=0$.
Suppose that there are five vertices $v^1,\ldots,v^5$ of $T_n$ inducing $C_5$ (listed in any order) and
let $i$ be the smallest index such that two of the vertices differ in the $i$-th coordinate.
By symmetry, we may assume $v^1_i=1$ and $v^2_i=2$.
One of the vertices must have the $i$-th coordinate equal to $3$ (otherwise, the induced tournament is not strongly connected);
hence, we can also assume that $v^3_i=3$.
By symmetry, we may assume that $v^4_i=1$ and $v^5_i\in\{1,2\}$.
If $v^5_i=1$, then the in-degree of $v_2$ in the subtournament induced by $v^1,\ldots,v^5$ is three, which is impossible.
If $v^5_i=2$, then both $v^1$ and $v^4$ are in-neighbors of each of the vertices $v^2$ and $v^5$, and
since there is an edge $v^2v^5$, the in-degree of either $v^2$ or $v^5$ is three, which is not possible.
We conclude that no five vertices of $T_n$ induce $C_5$ and so $t(C_5,T_n)=0$.
\end{proof}

\section{The Sidorenko property for regular $W$}
\label{sec:Sidorenko}

In this section, we prove our main result,
which complements the constructions presented in Section~\ref{sec:constr}
by showing that every tournament $T[a,b,c]$ with $a+b+c\ge 4$ is quasirandom forcing in regular tournamentons.
We start with the following auxiliary lemma on digraphs $B[k]$;
recall that $B[k]$ is the digraph obtained from $C[1,1,k]$
by removing the edge between the two vertices contained in the parts of size one.

\begin{lemma}
\label{lm:forcing}
Let $k\ge 2$.
Every regular tournamenton $W$ satisfies that $t(B[k],W)\ge 2^{-2k}$ and
the equality holds if and only if $W\equiv 1/2$.
\end{lemma}

\begin{proof}
Fix $k\ge 2$ and a regular tournamenton $W$.
We define two auxiliary functions $F:[0,1]^2\to [0,1]$ and $G:[0,1]^2\to [0,1]$ as follows:
\begin{align*}
F(x,y) & = \int W(x,z)W(y,z)\dd z\mbox{ and} \\
G(x,y) & = \int W(x,z)W(z,y)\dd z.
\end{align*}
Informally speaking,
$F(x,y)$ measures the number of common out-neighbors of $x$ and $y$ and
$G(x,y)$ measures the number of paths of length two from $x$ to $y$.
Observe that it holds for every $x,y\in [0,1]$ that
\begin{equation}
F(x,y)+G(x,y) = \int W(x,z)W(y,z)+W(x,z)W(z,y)\dd z = \int W(x,z)\dd z = \frac{1}{2}.
\label{eq:FGsum}
\end{equation}
Since the tournamenton $W$ is regular, it holds that
\[\int F(x,y)\dd x\dd y=t(S^-[2],W)=\frac{1}{4},\]
which implies using \eqref{eq:FGsum} that
\[\int G(x,y)\dd x\dd y=\frac{1}{4}.\]
Jensen's Inequality now yields that
\[t(B[k],W)=\int G(x,y)^k\dd x\dd y\ge\left(\int G(x,y)\dd x\dd y\right)^k=2^{-2k},\]
and the equality holds if and only if $G(x,y)=1/4$ for almost all $(x,y)\in [0,1]^2$.
By \eqref{eq:FGsum}, it holds that $t(B[k],W)= 2^{-2k}$ if and only if
$F(x,y)=1/4$ for almost all $(x,y)\in [0,1]^2$.
However, this is exactly the limit formulation of the property $P_4$ of quasirandom sequences of tournaments from~\cite{ChuG91}.
\end{proof}

To state the next lemma, we need to define three auxiliary functions, which will be parameterized by $k\in\NN$.
The functions $N_{W,k}^{+}:[0,1]^k\to [0,1]$ and $N_{W,k}^-:[0,1]^k\to [0,1]$
measure the ``size'' of the common out-neighborhood and the common in-neighborhood of a $k$-tuple points in a tournamenton $W$:
\begin{align*}
N_{W,k}^+(x_1,\ldots,x_k) & = \int \prod_{i\in [k]} W(x_i,z)\dd z\mbox{ and}\\
N_{W,k}^-(x_1,\ldots,x_k) & = \int \prod_{i\in [k]} W(z,x_i)\dd z.
\end{align*}
Finally, the function $D_{W,k}:[0,1]^k\to [0,1]$ measures the density of edges
directed from the common out-neighborhood to the common in-neighborhood:
\[
D_{W,k}(x_1,\ldots,x_k)=\frac{\int W(y,z)\prod\limits_{i\in [k]} W(x_i,y)W(z,x_i)\dd y\dd z}{N_{W,k}^+(x_1,\ldots,x_k)N_{W,k}^-(x_1,\ldots,x_k)}.
\]
When the tournamenton $W$ is clear from context, we drop it from the subscript and
simply write $N_k^+$, $N_k^-$ and $D_k$.
We will also write $x_{[k]}$ instead of $x_1,\ldots,x_k$.

Observe that the following holds for every $k\in\NN$ and every tournamenton $W$:
\begin{equation}
t\left(C[1,1,k],W)\right)=\int N_{W,k}^+(x_{[k]})D_{W,k}(x_{[k]})N_{W,k}^-(x_{[k]}) \dd x_{[k]}.
\label{eq:C11k}
\end{equation}

The next lemma says that the digraphs $C[a,b,k]$ have a Sidorenko type property:

\begin{lemma}
\label{lm:Cabk}
The following holds for every tournamenton $W$ and all $a,b,k\in\NN$:
\begin{equation}
t\left(C[a,b,k],W\right)\ge \int N_{W,k}^+(x_{[k]})^a D_{W,k}(x_{[k]})^{ab} N_{W,k}^-(x_{[k]})^b \dd x_{[k]}
\label{eq:Cabk}
\end{equation}
\end{lemma}

\begin{proof}
The proof of the lemma follows the standard argument for the Sidorenko property of bipartite graphs.
Throughout the proof, we use $W(x_{[k]}, y_{[\ell]})$ as
a shorthand notation for the double product $\prod_{i \in [k]}\prod_{j \in [\ell]}W(x_i, y_j)$
where $x \in [0, 1]^k$ and $y \in [0, 1]^\ell$.
We extend the notation to the case when $k=1$ or $\ell=1$,
e.g., $W(x_{[k]}, y)$ stands for the product $\prod_{i \in [k]}W(x_i, y)$.

Fix $a,b,k\in\NN$ and a tournamenton $W$ for the proof.
Consider $x\in [0,1]^k$ such that $N_k^+(x_{[k]})>0$ and $N_k^-(x_{[k]})>0$.
Observe that it holds that
    \[\int\frac{W(x_{[k]},y_{[a]})}{N_k^+(x_{[k]})^a}\dd y_{[a]}=
    \prod_{j\in [a]}\int\frac{W(x_{[k]},y_j)}{N_k^+(x_{[k]})}\dd y_j=1.\]
In particular, we can interpret the integrand in the left integral as the density of a probability measure on $[0,1]^a$.
We now apply Jensen's Inequality as follows:
\begin{eqnarray}
    & & \int \frac{W(x_{[k]},y_{[a]})W(y_{[a]},z_{[b]})W(z_{[b]},x_{[k]})}{N_k^+(x_{[k]})^a N_k^-(x_{[k]})^b}\dd y_{[a]} \dd z_{[b]}\nonumber\\
& = & \int\frac{W(x_{[k]},y_{[a]})}{N_k^+(x_{[k]})^a}\int\frac{W(y_{[a]},z_{[b]})W(z_{[b]},x_{[k]})}{N_k^-(x_{[k]})^b}\dd z_{[b]}\dd y_{[a]} \nonumber\\
& = & \int\frac{W(x_{[k]},y_{[a]})}{N_k^+(x_{[k]})^a}\left(\int\frac{W(y_{[a]},z)W(z,x_{[k]})}{N_k^-(x_{[k]})}\dd z\right)^b\dd y_{[a]}\nonumber\\
& \ge &\left(\int\frac{W(x_{[k]},y_{[a]})}{N_k^+(x_{[k]})^a}\int\frac{W(y_{[a]},z)W(z,x_{[k]})}{N_k^-(x_{[k]})}\dd z\dd y_{[a]}\right)^b\nonumber\\
& = & \left(\int \frac{W(x_{[k]},y_{[a]})W(y_{[a]},z)W(z,x_{[k]})}{N_k^+(x_{[k]})^a N_k^-(x_{[k]})}\dd y_{[a]}\dd z\right)^b.
\label{eq:Cabk1}
\end{eqnarray}
Next observe that it holds that
    \[\int\frac{W(z,x_{[k]})}{N_k^-(x_{[k]})}\dd z=1,\]
which means that we can interpret the integrand as the density of a probability measure on $[0,1]$.
So, we get by another application of Jensen's Inequality the following:
\begin{eqnarray}
    & & \int \frac{W(x_{[k]},y_{[a]})W(y_{[a]},z)W(z,x_{[k]})}{N_k^+(x_{[k]})^a N_k^-(x_{[k]})}\dd y_{[a]}\dd z\nonumber\\
    & = & \int \frac{W(z,x_{[k]})}{N_k^-(x_{[k]})}\left(\int\frac{W(x_{[k]},y)W(y,z)}{N_k^+(x_{[k]})}\dd y\right)^a\dd z\nonumber\\
    & \ge & \left(\int \frac{W(z,x_{[k]})}{N_k^-(x_{[k]})}\int\frac{W(x_{[k]},y)W(y,z)}{N_k^+(x_{[k]})}\dd y\dd z\right)^a\nonumber\\
    & = & \left( \int \frac{W(x_{[k]},y)W(y,z)W(z,x_{[k]})}{N_k^+(x_{[k]})N_k^-(x_{[k]})}\dd y\dd z\right)^a=D_k(x_{[k]})^a.
\label{eq:Cabk2}
\end{eqnarray}

    Let $\Omega\subseteq [0,1]^k$ be the set of those $x \in [0, 1]^k$ such that $N_k^+(x_{[k]})>0$ and $N_k^-(x_{[k]})>0$. Using
    \eqref{eq:Cabk1} and \eqref{eq:Cabk2} we conclude that for every $x \in \Omega$ it holds that

\begin{eqnarray}
    & & \int W(x_{[k]},y_{[a]})W(y_{[a]},z_{[b]})W(z_{[b]},x_{[k]}) \dd y_{[a]} \dd z_{[b]}\nonumber\\
    & = & N_k^+(x_{[k]})^aN_k^-(x_{[k]})^b \int\frac{W(x_{[k]},y_{[a]})W(y_{[a]},z_{[b]})W(z_{[b]},x_{[k]})}{N_k^+(x_{[k]})^aN_k^-(x_{[k]})^b} \dd y_{[a]} \dd z_{[b]}\nonumber\\
    & \ge & N_k^+(x_{[k]})^a N_k^-(x_{[k]})^b\left(\int \frac{W(x_{[k]},y_{[a]})W(y_{[a]},z)W(z,x_{[k]})}{N_k^+(x_{[k]})^a N_k^-(x_{[k]})}\dd y_{[a]}\dd z\right)^b \nonumber\\
    & \ge & N_k^+(x_{[k]})^a N_k^-(x_{[k]})^bD_k(x_{[k]})^{ab}.
    \label{eq:Cabk3}
\end{eqnarray}

    Using that inequality (\ref{eq:Cabk3}) holds for every $x \in \Omega$, we estimate the homomorphism density of $C[a,b,k]$ as follows:
\begin{align*}
    t\left(C[a,b,k],W\right)
    & =    \int W(x_{[k]},y_{[a]})W(y_{[a]},z_{[b]})W(z_{[b]},x_{[k]}) \dd x_{[k]} \dd y_{[a]} \dd z_{[b]}\\
    & \ge  \int\limits_{\Omega}W(x_{[k]},y_{[a]})W(y_{[a]},z_{[b]})W(z_{[b]},x_{[k]}) \dd x_{[k]} \dd y_{[a]} \dd z_{[b]}\\
    & \ge  \int\limits_{\Omega}N_k^+(x_{[k]})^a D_k(x_{[k]})^{ab}N_k^-(x_{[k]})^b\dd x_{[k]}.
\end{align*}
    Since it holds that $N_k^+(x_{[k]})^a D_k(x_{[k]})^{ab} N_k^-(x_{[k]})^b=0$ for every $x \in [0,1]^k\setminus\Omega$,
we obtain that
    \[t\left(C[a,b,k],W\right)\ge\int N_k^+(x_{[k]})^a D_k(x_{[k]})^{ab} N_k^-(x_{[k]})^b\dd x_{[k]}.\]
Since the choice of $a,b,k\in\NN$ and a tournamenton $W$ was arbitrary, the proof of the lemma is completed.
\end{proof}

We are now ready to prove the key theorem of this section,
which we then use to prove Theorem~\ref{thm:Tabc}.

\begin{theorem}
\label{thm:Cabc}
Let $a,b,c\in\NN$ such that $a+b+c\ge 4$.
Every regular tournamenton $W$ satisfies that
\begin{equation}
t\left(C[a,b,c],W\right)\ge 2^{-ab-ac-bc}.
\label{eq:main}
\end{equation}
Moreover, the equality in \eqref{eq:main} holds if and only if $W\equiv 1/2$.
\end{theorem}

\begin{proof}
Fix $a,b,c\in\NN$ such that $a+b+c\ge 4$ and a regular tournamenton $W$.
Note that \[t\left(C[a,b,c],W\right)=t\left(C[b,c,a],W\right)=t\left(C[c,a,b],W\right);\]
so we can assume by this rotational symmetry that $a$ is the smallest among $a$, $b$ and $c$, and
if the smallest value is not unique among $a$, $b$ and $c$, then it additionally holds that $a=b$.
Observe that $c\ge 2$ (if $c=1$, then $a=b=1$, which is impossible as $a+b+c\ge 4$).

As in Lemma~\ref{lm:forcing}, let $G:[0,1]^2\to [0,1]$ be defined as
\[G(x,y)=\int W(x,z)W(z,y)\dd z;\]
informally speaking, $G(x,y)$ is the density of directed paths from $x$ to $y$ of length two.
Also note that $G$ is the square of $W$ in the operator sense.
We observe that the regularity of $W$ implies that
the function $G$ is symmetric, i.e., it holds that
\begin{eqnarray*}
G(y,x) &=& \int W(y,z)W(z,x)\dd z \\
         &=& \int (1-W(z,y))(1-W(x,z))\dd z \\
	 &=& 1 -\int W(x,z)\dd z-\int W(z,y)\dd z+\int W(x,z)W(z,y)\dd z \\
	 &=& 1-\frac{1}{2}-\frac{1}{2}+\int W(x,z)W(z,y)\dd z=G(x,y)
\end{eqnarray*}
Recall that $B[c]$ is the digraph obtained from $C[1,1,c]$
by removing the edge between the two vertices contained in the parts of size one.
The definition of $G$ yields that
\begin{eqnarray*}
t\left(B[c],W\right) & = & \int G(x,y)^c\dd x\dd y \mbox{ and}\\
t\left(C[1,1,c],W\right) & = & \int G(x,y)^c W(y,x) \dd x\dd y.
\end{eqnarray*}
Since the function $G(x,y)$ is symmetric and $W(x,y)+W(y,x)=1$ for all $(x,y)\in [0,1]^2$,
we obtain that
\begin{eqnarray}
t\left(C[1,1,c],W\right) & = & \frac{1}{2} t\left(C[1,1,c],W\right) + \frac{1}{2} t\left(C[1,1,c],W\right) \nonumber \\
                         & = & \frac{1}{2} \int G(x,y)^c W(y,x) \dd x \dd y + \frac{1}{2} \int G(y,x)^c W(x,y)\dd x \dd y \nonumber \\
			 & = & \frac{1}{2} \int G(x,y)^c W(y,x) +  G(y,x)^c W(x,y)\dd x \dd y \nonumber \\
			 & = & \frac{1}{2} \int G(x,y)^c W(y,x) +  G(x,y)^c W(x,y)\dd x \dd y \nonumber \\
			 & = & \frac{1}{2} \int G(x,y)^c\dd x\dd y = \frac{t\left(B[c],W\right)}{2}.
\label{eq:CBhalf}
\end{eqnarray}

Recall the definitions of the functions $N_c^+:[0,1]^c\to [0,1]$, $N_c^-:[0,1]^c\to [0,1]$ and $D_c : [0,1]^c \to [0,1]$
given before the statement of Lemma~\ref{lm:Cabk}.
Also recall that $S^+[c]$ is the directed $c$-leaf star with the center vertex being the source and
$S^-[c]$ is the directed $c$-leaf star with the center vertex being the sink.
Observe that the following four identities hold (we use that $W$ is regular in the last two):
\begin{eqnarray}
    t\left(B[c],W\right) & = & \int N_c^+(x_{[c]})N_c^-(x_{[c]})\dd x_{[c]}, \label{eq:Bc} \\
t\left(C[1,1,c],W\right) & = & \int N_c^+(x_{[c]})D_c(x_{[c]})N_c^-(x_{[c]})\dd x_{[c]}, \label{eq:C11c} \\
t\left(S^+[c],W\right) & = & 2^{-c} = \int N_c^-(x_{[c]})\dd x_{[c]}\mbox{ and} \label{eq:Nc-} \\
t\left(S^-[c],W\right) & = & 2^{-c} = \int N_c^+(x_{[c]})\dd x_{[c]}. \label{eq:Nc+}
\end{eqnarray}

We now apply generalized H\"older's Inequality with $p_1=ab$, $p_2=\frac{ab}{ab-b}$ and $p_3=\frac{ab}{b-a}$ (recall that $a\le b$)
to obtain the following inequality:
\begin{eqnarray}
\int N_c^+(x_{[c]})D_c(x_{[c]})N_c^-(x_{[c]})\dd x_{[c]}
& \le & \left(\int N_c^+(x_{[c]})^a D_c(x_{[c]})^{ab} N_c^-(x_{[c]})^b\dd x_{[c]}\right)^{\frac{1}{ab}} \times \nonumber \\
& & \left(\int N_c^+(x_{[c]})N_c^-(x_{[c]})\dd x\right)^{\frac{ab-b}{ab}} \times \nonumber \\
& & \left(\int N_c^+(x_{[c]})\dd x\right)^{\frac{b-a}{ab}}.
\label{eq:Holder1}
\end{eqnarray}
We plug \eqref{eq:Bc}, \eqref{eq:C11c} and \eqref{eq:Nc+} to \eqref{eq:Holder1} and get that
\begin{eqnarray*}
t\left(C[1,1,c],W\right)
& \le & \left(\int N_c^+(x_{[c]})^a D_c(x_{[c]})^{ab} N_c^-(x_{[c]})^b\dd x_{[c]}\right)^{\frac{1}{ab}} \times \\
& & t\left(B[c],W\right)^{\frac{ab-b}{ab}} \times 2^{\frac{-c(b-a)}{ab}},
\end{eqnarray*}
which yields by Lemma~\ref{lm:Cabk} that
\begin{equation}
t\left(C[1,1,c],W\right) \le
t\left(C[a,b,c],W\right)^{\frac{1}{ab}} \times
t\left(B[c],W\right)^{\frac{ab-b}{ab}} \times 2^{\frac{-c(b-a)}{ab}}.
\label{eq:Holder2}
\end{equation}
The inequality \eqref{eq:Holder2} is equivalent to
\begin{equation}
t\left(C[a,b,c],W\right)
\ge
\frac{t\left(C[1,1,c],W\right)^{ab}}{2^{-c(b-a)}\times t\left(B[c],W\right)^{ab-b}}.
\label{eq:Holder3}
\end{equation}
Using \eqref{eq:CBhalf} and Lemma~\ref{lm:forcing},
we manipulate the right side of \eqref{eq:Holder3} as follows:
\[\frac{t\left(C[1,1,c],W\right)^{ab}}{2^{-c(b-a)}\times t\left(B[c],W\right)^{ab-b}}=
  \frac{2^{-ab}\times t\left(B[c],W\right)^b}{2^{ac-bc}}
  \ge\frac{2^{-ab-2bc}}{2^{ac-bc}}
   =2^{-ac-ab-bc};
\]
note that Lemma~\ref{lm:forcing} yields that the equality above holds if and only if $W\equiv 1/2$.
We conclude that
\[2^{-ac-ab-bc}\le t\left(C[a,b,k],W\right)\]
and the equality holds if and only if $W\equiv 1/2$.
\end{proof}

Before proving our main theorem, we need an auxiliary lemma.
Informally speaking, we use that transitive tournaments have the Sidorenko property
to show that adding a transitive tournament on a vertex set formed by twins in a digraph $H$
drops density of $H$ by at most the expected homomorphism density of the added transitive tournament in a random tournament.

\begin{lemma}
\label{lm:trans}
Let $H$ be a digraph that contains an independent set $A$ such that all vertices in $A$ are twins, and
let $H'$ be the digraph obtained from $H$ by adding the $|A|$-vertex transitive tournament on $A$.
It holds that
\[t(H',W)\ge 2^{-\binom{|A|}{2}}\;t(H,W).\]
for every tournamenton $W$.
\end{lemma}

\begin{proof}
The proof will proceed by induction on the size of the set $A$.
Before presenting the proof, we introduce some notation.
Consider a digraph $H$, an independent set $A$ as in the statement of the lemma and let $a \in A$ be any vertex of $A$.
Let $B$ be the set of the vertices of $H$ not contained in $A$.
For a tournamenton $W$,
we define two functions $F_W:[0,1]^B\times [0,1]\to [0,1]$ and $G_W:[0,1]^B\to [0,1]$ as follows:
\begin{align*}
    F_W(x_{B},z) & =\prod_{va\in E(H) \cap (B \times \{a\})}W(x_v,z)\prod_{av\in E(H) \cap (\{a\} \times B)}W(z,x_v)\mbox{ and}\\
G_W(x_{B}) & =\prod_{uv\in E(H)\cap (B\times B)} W(x_u,x_v).
\end{align*}
Note that the definition of $F_W(x_{B},z)$ does not depend on the choice of $a\in A$ as
all vertices contained in $A$ are twins.
Observe that
\[t(H,W)=\int G_W(x_B)\prod_{a\in A}F_W(x_B,z_a)\dd x_B\dd z_A\]
for any tournamenton $W$.

We are now ready to present the proof of the lemma,
in which we use the above introduced notation.

If $|A|=1$, then the statement holds trivially as $H=H'$.
We next analyze the case $|A|=2$.
The definition of $H'$ implies that
\[t(H',W)=\int G_W(x_B)F_W(x_B,z_1)F_W(x_B,z_2)W(z_1,z_2)\dd x_B\dd z_1\dd z_2.\]
Since the role of $z_1$ and $z_2$ in the above expression is symmetric we also have that
\[t(H',W)=\int G_W(x_B)F_W(x_B,z_1)F_W(x_B,z_2)W(z_2,z_1)\dd x_B\dd z_1\dd z_2,\]
which yields using the identity $W(z_1,z_2)+W(z_2,z_1)=1$ for all $(z_1,z_2)\in [0,1]^2$ that
\[2t(H',W)=\int G_W(x_B)F_W(x_B,z_1)F_W(x_B,z_2)\dd x_B\dd z_1\dd z_2=t(H,W).\]
This concludes the proof of the case $|A|=2$.
Note that we have proven that the inequality always holds with equality when $|A|=2$.

We now establish the induction step in the case $|A|\ge 3$.
Let $a_1,\ldots,a_{|A|}$ be the vertices of $A$ listed in the order that is consistent with the transitive tournament on $A$ in $H'$, and
let $H''$ be the digraph obtained from $H$ by adding an edge directed from $a_1$ to each of the vertices $a_2,\ldots,a_{|A|}$.
Note that the vertices of $A\setminus\{a_1\}$ are twins in $H''$.
By the induction hypothesis applied to the digraph $H''$ with the digraph $H'$ and the independent set $A \setminus \{a_1\}$,
we obtain that
\begin{equation}
t(H',W)\ge 2^{-\binom{|A|-1}{2}}\;t(H'',W).
\label{eq:trans1}
\end{equation}
We next apply the induction to the digraph $H$ and the digraph obtained from $H$ by adding the edge directed from $a_1$ to $a_2$,
i.e., we invoke the case when the size of the independent set is two and the inequality holds with equality to obtain that
\[\int G_W(x_B)\left(\prod_{a\in A}F_W(x_B,z_a)\right)W(z_{a_1},z_{a_2})\dd x_B\dd z_A=\frac{t(H,W)}{2}.\]
We now apply H\"older's Inequality to derive that
\begin{align*}
  &\frac{t(H,W)}{2}=\int G_W(x_B)\left(\prod_{a\in A}F_W(x_B,z_a)\right)W(z_{a_1},z_{a_2})\dd x_B\dd z_A\\
= & \int G_W(x_B)F_W(x_B,z_1)\left(\int F_W(x_B,z)W(z_1,z)\dd z\right)\left(\int F_W(x_B,z)\dd z\right)^{|A|-2}\dd x_B\dd z_1 \\
\le & \left(\int G_W(x_B)F_W(x_B,z_1)\left(\int F_W(x_B,z)W(z_1,z)\dd z\right)^{|A|-1}\dd x_B\dd z_1\right)^{\frac{1}{|A|-1}} \times \\
& \left(\int G_W(x_B)F_W(x_B,z_1)\left(\int F_W(x_B,z)\dd z\right)^{|A|-1}\dd x_B\dd z_1\right)^{\frac{|A|-2}{|A|-1}} \\
= & \; t(H'',W)^{\frac{1}{|A|-1}}\times t(H,W)^{\frac{|A|-2}{|A|-1}},
\end{align*}
which yields that
\begin{equation}
2^{-|A|+1}\;t(H,W) \le t(H'',W).
\label{eq:trans2}
\end{equation}
The inequalities \eqref{eq:trans1} and \eqref{eq:trans2} combine to
\[t(H',W)\ge 2^{-\binom{|A|-1}{2}-|A|+1}\;t(H,W)=2^{-\binom{|A|}{2}}\;t(H,W).\]
This completes the induction step.
\end{proof}

We are now ready to prove the main theorem of this section, which in combination with Theorem~\ref{thm:constr} readily yields Theorem~\ref{thm:main}.

\begin{theorem}
\label{thm:Tabc}
Let $a,b,c\in\NN$ such that $a+b+c\ge 4$.
Every regular tournamenton $W$ satisfies that
\[t(T[a,b,c],W)\ge 2^{-\binom{a+b+c}{2}},\]
and the equality holds if and only if $W\equiv 1/2$.
\end{theorem}

\begin{proof}
Fix $a,b,c\in\NN$ such that $a+b+c\ge 4$ and a regular tournamenton $W$.
If $W\equiv 1/2$, then $t(T[a,b,c],W)=2^{-\binom{a+b+c}{2}}$.
We assume that $W\not\equiv 1/2$ for the rest of the proof and
show that $t(T[a,b,c],W)>2^{-\binom{a+b+c}{2}}$.
Theorem~\ref{thm:Cabc} implies that
\begin{equation}
t(C[a,b,c],W)>2^{-ab-ac-bc}.
\label{eq:Tabc0}
\end{equation}
Let $T_1$ be the digraph obtained from $C[a,b,c]$ by adding the $a$-vertex transitive tournament on the part of size $a$ and
let $T_2$ be the digraph obtained from $T_1$ by adding the $b$-vertex transitive tournament on the part of size $b$;
observe that $T[a,b,c]$ is the digraph obtained from $T_2$ by adding the $c$-vertex transitive tournament on the part of size $c$.
Lemma~\ref{lm:trans} yields the following inequalities:
\begin{align*}
t(T_1,W) & \ge 2^{-\binom{a}{2}}\; t(C[a,b,c],W),\\
t(T_2,W) & \ge 2^{-\binom{b}{2}}\; t(T_1,W) \mbox{ and}\\
t(T[a,b,c],W) & \ge 2^{-\binom{c}{2}}\; t(T_2,W).
\end{align*}
We now combine these three inequalities with \eqref{eq:Tabc0} to obtain that
\begin{align*}
t(T[a,b,c],W) & \ge 2^{-\binom{a}{2}-\binom{b}{2}-\binom{c}{2}}\;t(C[a,b,c],W)\\
              & >2^{-\binom{a}{2}-\binom{b}{2}-\binom{c}{2}-ab-ac-bc}=2^{-\binom{a+b+c}{2}},
\end{align*}
which completes the proof of the theorem.
\end{proof}

\section{The anti-Sidorenko property for regular $W$}
\label{sec:antiSidorenko}

In this section, we give a proof of Theorem~\ref{thm:max} using an argument similar to the one in~\cite{BucLSS21} that showed there are only finitely many anti-Sidorenko tournaments in the setting of arbitrary tournament sequences.
The following Proposition is a direct translation of Theorem~\ref{thm:max} to the language of combinatorial limits.

\begin{proposition}
\label{prop:max}
For every $n$-vertex tournament $H$ with $n\ge 10$ there exists a regular tournamenton $W$ such that
\[t(H,W)>2^{-\binom{n}{2}}.\]
\end{proposition}

\begin{figure}
\begin{center}
\epsfbox{qitourn-9.mps}
\end{center}
\caption{The tournamenton $W$ from the proof of Proposition~\ref{prop:max}
         when $H$ is the tournament $W_4$ depicted in Figure~\ref{fig:special}.
	 The origin of the coordinate system is in the top left corner,
	 the division between the parts as defined in the proof is visualized by dotted lines, and
	 the division between the parts $A_i$'s and $B_i$'s by dashed lines.}
\label{fig:max}
\end{figure}

\begin{proof}
Fix a tournament $H$ with $n$ vertices $v_1,\ldots,v_n$.
We construct a regular tournamenton $W$ that satisfy the inequality given in the proposition.
Split the interval $[0,1]$ into $2n$ parts $A_1,\ldots,A_n$ and $B_1,\ldots,B_n$, each of measure $1/2n$.
The tournamenton $W$ is defined as follows:
\[
W(x,y)=\begin{cases}
       1 & \mbox{if $x\in A_i$, $y\in A_j$ and $v_iv_j \in E(H)$,} \\
       1 & \mbox{if $x\in B_i$, $y\in B_j$ and $v_iv_j \in E(H)$,} \\
       1 & \mbox{if $x\in A_i$, $y\in B_j$ and $v_jv_i \in E(H)$,} \\
       1 & \mbox{if $x\in B_i$, $y\in A_j$ and $v_jv_i \in E(H)$,} \\
       1/2 & \mbox{if $x\in A_i\cup B_i$ and $y\in A_i\cup B_i$, and} \\
       0 & \mbox{otherwise.}
       \end{cases}
\]
An example of the tournamenton $W$ when $H$ is the tournament $W_4$ from Figure~\ref{fig:special}
is given in Figure~\ref{fig:max}.
Observe that if either $x_{v_i}\in A_i$ for all $i\in [n]$ or $x_{v_i}\in B_i$ for all $i\in [n]$,
then the integrand in \eqref{eq:tHW} is equal to one.
It follows that
\[t(H,W)\ge 2\cdot (2n)^{-n}.\]
Since $2\cdot (2n)^{-n}>2^{-\binom{n}{2}}$ for $n\ge 10$, the statement of the proposition follows.
\end{proof}

\section{Conclusion}
\label{sec:concl}

In this paper, we characterized all the tournaments $H$ with the Sidorenko property for regular tournamentons.
Regarding any tournament $H$ with the anti-Sidorenko property for regular tournamentons, we proved that the number of vertices of $H$ must be at most nine.

We have computationally identified several tournaments (in addition to those
found by Noel, Ranganathan and Simbaqueba~\cite{NoeRS25}) with this property;
three such tournaments on six vertices are depicted in Figure~\ref{fig:T6}.
However, we do not have a conjectured list of all tournaments with the anti-Sidorenko property for regular tournamentons,
and we leave their characterization as an~open problem.

\begin{problem}
\label{prob:0}
Characterize which tournaments $H$ satisfy that
the constant tournamenton is the unique maximizer of $t(H,W)$ among regular tournamentons $W$.
\end{problem}

Note that a resolution of Problem~\ref{prob:0} together with Theorem~\ref{thm:main}
would yield a full solution of \cite[Problem 6.1]{NoeRS25}.

\begin{figure}
\begin{center}
\epsfbox{qitourn-3.mps}
\end{center}
\caption{Three $6$-vertex tournaments $H$ that we have computationally verified to have the anti-Sidorenko property for regular tournamentons.}
\label{fig:T6}
\end{figure}

\subsection{The Sidorenko property via entropy}
The entropy method has recently gained a prominent role in making progress on various problems in extremal combinatorics,
particularly Sidorenko type problems, see e.g.~\cite{BehMN24, ConKLL18b, ConKLL18, ConL17, Fit18, GrzLLV22, Lee21, LiS11, Par14, Sze14, ChaY24, ChaY24b, ChaY24c, ChaALY25}.
We found an alternative proof of Theorem~\ref{thm:Cabc} using the entropy method when $c=1$ and $a$ and $b$ are arbitrary
but we were unable to extend it to full generality.
Still, we want to sketch the argument.
We will assume that the reader is familiar with basic concepts concerning the use of the entropy method and
we refer to e.g.~\cite{Gal14,Gow15}
for the exposition in the setting of combinatorics.

Fix $a,b\in\NN$ and let $G$ be an $n$-vertex regular tournament.
It is well-known that $G$ has $\frac{n^3-n}{24}=\frac{1}{4}\binom{n}{3}+O(n^2)$ cyclically oriented triangles and
every vertex of $G$ is in exactly $\frac{n^2-1}{8}$ of such triangles,
i.e., every vertex is in the same number of cyclically oriented triangles.
We will estimate the number of homomorphisms from $T[a,b,1]$ to $G$.
Let $(x,y,z)$ be a cyclically oriented triangle of $G$ chosen uniformly at random so that $xy$, $yz$ and $zx$ are edges in $G$. Note that we consider $(x,y,z)$ and $(y,z,x)$ to be different as
we wish to count homomorphisms.
Note that $H(x,y,z)=\log\frac{n^3-n}{8}$ and
\[H(x,y)=H(y,z)=H(z,x)\le\log{n \choose 2}\le\log\frac{n^2}{2};\]
the first inequality in the displayed estimate follows from the fact that
the entropy is maximum for the uniform distribution on the ${n \choose 2}$ edges of $G$.
We now sample $(y,z)$ according to the marginal distribution coming from $(x,y,z)$ and
sample $x_1,\ldots,x_a$ as conditionally independent copies of $x$ given $(y, z)$.
In this way, we obtain a distribution on $(a+2)$-tuples $(x_1,\ldots,x_a,y,z)$ that
corresponds to homomorphisms from $C[a,1,1]$ to $G$.
Note that
\begin{equation}
H(x_1,\ldots,x_a,y,z)=aH(x| y, z)+H(y, z). \label{eq:concl1}
\end{equation}
We now sample $(x_1,\ldots,x_a,z)$ according to the marginal distribution
coming from the distribution of $(x_1,\ldots,x_a,y,z)$ and sample $y_1,\ldots,y_b$ as conditionally independent copies of $y$ given $(x_1, \ldots, x_a, z)$.
In this way, we obtain a distribution on $(a+b+1)$-tuples $(x_1,\ldots,x_a,y_1,\ldots,y_b,z)$ that
corresponds to homomorphisms from $C[a,b,1]$ to $G$ and
we can compute its entropy as follows:
\begin{equation}
H(x_1,\ldots,x_a,y_1,\ldots,y_b,z)=bH(y|x_1,\ldots,x_a,z)+H(x_1,\ldots,x_a,z). \label{eq:concl2}
\end{equation}
Since the $(a + 1)$-tuple $(x_1,\ldots,x_a,z)$ always induces the directed $a$-leaf star with the center vertex being the source and
the tournament $G$ is regular, we obtain that
\begin{equation}
H(x_1,\ldots,x_a,z)\le\log n\left(\frac{n-1}{2}\right)^a\le\log\frac{n^{a+1}}{2^a}.\label{eq:concl3}
\end{equation}
We now combine \eqref{eq:concl1}, \eqref{eq:concl2} and \eqref{eq:concl3} to obtain the following estimate
on the entropy of the constructed distribution on $(a+b+1)$-tuples $(x_1,\ldots,x_a,y_1,\ldots,y_b,z)$:
\begin{align*}
    H(x_1,&\ldots,x_a,y_1,\ldots,y_b,z) \\
& = b H(y|x_1,\ldots,x_a,z)+H(x_1,\ldots,x_a,z) \\
& = b H(x_1,\ldots,x_a,y,z) - (b-1) H(x_1,\ldots,x_a,z) \\
& = ab H(x|y,z) + b H(y,z) - (b-1) H(x_1,\ldots,x_a,z) \\
& = ab H(x,y,z) - (a-1)b H(y,z) - (b-1) H(x_1,\ldots,x_a,z) \\
&\ge ab\log\frac{n^3-n}{8} - (a-1)b\log\frac{n^2}{2}-(b-1) \log\frac{n^{a+1}}{2^a} \\
& = \log\frac{(n^3-n)^{ab}}{2^{3ab}}-\log\frac{n^{2ab-2b}}{2^{(a-1)b}}-\log\frac{n^{ab-a+b-1}}{2^{a(b-1)}}\\
& = \log\frac{n^{a+b+1}-O(n^{a+b-1})}{2^{ab+a+b}}.
\end{align*}
Comparing the entropy of the constructed distribution with the entropy of the uniform distribution yields that the number homomorphisms from $C[a, b, 1]$ to $G$ is at least $\frac{n^{a+b+1}(1+o(1))}{2^{ab+a+b}}$,
which implies that $t(C[a,b,1],W)\ge 2^{-ab-a-b}$ for any regular tournamenton $W$.
The argument presented above can be extended to show that the equality holds if and only if $W\equiv 1/2$.

\bibliographystyle{bibstyle}
\bibliography{qitourn}

@article {BolH90,
  title={Powers of {H}amilton cycles in tournaments},
  author={Bollob\'as, B\'ela and H\"agghkvist, Roland},
  journal={Journal of Combinatorial Theory, Series B},
  volume={50},
  year={1990},
  pages={309--318}
  }

@article {ErdS83,
    AUTHOR = {Erd\H{o}s, Paul and Simonovits, Mikl\'{o}s},
     TITLE = {Supersaturated graphs and hypergraphs},
   JOURNAL = {Combinatorica},
    VOLUME = {3},
      YEAR = {1983},
     PAGES = {181--192},
}

@article {Sid89,
    AUTHOR = {Sidorenko, A. F.},
     TITLE = {Cycles in graphs and functional inequalities},
   JOURNAL = {Akademiya Nauk SSSR. Matematicheskie Zametki},
    VOLUME = {46},
      YEAR = {1989},
     PAGES = {72--79, 104}
}

@article {Sid91,
    AUTHOR = {Sidorenko, A. F.},
     TITLE = {Inequalities for functionals generated by bipartite graphs},
   JOURNAL = {Diskretnaya Matematika},
    VOLUME = {3},
      YEAR = {1991},
     PAGES = {50--65}
}

@article{BehMN24,
  title={Off-diagonal commonality of graphs via entropy},
  author={Behague, Natalie and Morrison, Natasha and Noel, Jonathan A},
  journal={SIAM Journal on Discrete Mathematics},
  volume={38},
  number={3},
  pages={2335--2360},
  year={2024},
  publisher={SIAM}
}

@article {GrzIKK23,
    author={Grzesik, Andrzej and I{\v{l}}kovi{\v{c}}, Daniel and Kielak, Bart{\l}omiej and Kr{\'a}l’, Daniel},
     TITLE = {Quasirandom-forcing orientations of cycles},
   JOURNAL = {SIAM Journal on Discrete Mathematics},
      YEAR = {2023},
    NUMBER = {4},
     PAGES = {2689--2716}
}

@Article{ChuGW89,
  author   = {Chung, F. R. K. and Graham, R. L. and Wilson, R. M.},
  title    = {Quasi-random graphs},
  journal  = {Combinatorica},
  year     = {1989},
  volume   = {9},
  number   = {4},
  pages    = {345--362},
  issn     = {0209-9683}
}

@Article{Tho87,
  author    = {Thomason, Andrew},
  title     = {Pseudo-random graphs},
  journal   = {Annals of Discrete Mathematics},
  year      = {1987},
  volume    = {144},
  pages     = {307--331},
}

@InCollection{Tho87b,
  author    = {Thomason, Andrew},
  title     = {Random graphs, strongly regular graphs and pseudo-random graphs},
  booktitle = {Surveys in Combinatorics},
  publisher = {Cambridge Univ. Press},
  year      = {1987},
  volume    = {123},
  series    = {London Mathematical Society Lecture Note Series},
  pages     = {173--196},
}

@Article{Rod86,
  author   = {Vojt\v{e}ch R\"odl},
  title    = {On universality of graphs with uniformly distributed edges},
  journal  = {Discrete Mathematics},
  year     = {1986},
  volume   = {59},
  number   = {1},
  pages    = {125--134},
  issn     = {0012-365X},
  doi      = {https://doi.org/10.1016/0012-365X(86)90076-2},
  url      = {http://www.sciencedirect.com/science/article/pii/0012365X86900762},
}

@Article{Gow08,
  author    = {Gowers, W. T.},
  title     = {Quasirandom Groups},
  journal   = {Combinatorics, Probability and Computing},
  year      = {2008},
  volume    = {17},
  number    = {3},
  pages     = {363--387},
  doi       = {10.1017/S0963548307008826},
  publisher = {Cambridge University Press},
}

@Article{KraP13,
  author     = {Kr\'{a}l', Daniel and Pikhurko, Oleg},
  title      = {Quasirandom permutations are characterized by 4-point densities},
  journal    = {Geometric and Functional Analysis},
  year       = {2013},
  volume     = {23},
  number     = {2},
  pages      = {570--579},
  issn       = {1016-443X},
  doi        = {10.1007/s00039-013-0216-9},
  mrclass    = {05A05 (60C05)},
  mrnumber   = {3053756},
  mrreviewer = {Dmitry A. Shabanov},
}

@Article{Gow07,
  author     = {Gowers, W. T.},
  title      = {Hypergraph regularity and the multidimensional {S}zemer\'{e}di theorem},
  journal    = {Annals of Mathematics, Second Series},
  year       = {2007},
  volume     = {166},
  number     = {3},
  pages      = {897--946},
  issn       = {0003-486X},
  doi        = {10.4007/annals.2007.166.897},
  mrclass    = {05D10 (05C30 05C35 05C55 05C80 28D15)},
  mrnumber   = {2373376},
  mrreviewer = {G\'{a}bor N. S\'{a}rk\"{o}zy},
}

@Article{Gow06,
  author     = {Gowers, W. T.},
  title      = {Quasirandomness, counting and regularity for 3-uniform hypergraphs},
  journal    = {Combinatorics, Probability and Computing},
  year       = {2006},
  volume     = {15},
  pages      = {143--184},
  issn       = {0963-5483},
  doi        = {10.1017/S0963548305007236},
  mrclass    = {05D05 (05C35 05C65)},
  mrnumber   = {2195580},
  mrreviewer = {J\'{o}zsef Solymosi},
}

@Article{Coo04,
  author     = {Cooper, Joshua N.},
  title      = {Quasirandom permutations},
  journal    = {Journal of Combinatorial Theory, Series A},
  year       = {2004},
  volume     = {106},
  number     = {1},
  pages      = {123--143},
  issn       = {0097-3165},
  doi        = {10.1016/j.jcta.2004.01.006},
  mrclass    = {05A05 (60C05)},
  mrnumber   = {2050120},
  mrreviewer = {Heinrich Niederhausen},
}

@Article{KohRS02,
  author     = {Kohayakawa, Yoshiharu and R\"{o}dl, Vojt\v{e}ch and Skokan, Jozef},
  title      = {Hypergraphs, quasi-randomness, and conditions for regularity},
  journal    = {Journal of Combinatorial Theory, Series A},
  year       = {2002},
  volume     = {97},
  number     = {2},
  pages      = {307--352},
  issn       = {0097-3165},
  doi        = {10.1006/jcta.2001.3217},
  mrclass    = {05C65 (05C80)},
  mrnumber   = {1883869},
  mrreviewer = {A. G. Thomason},
}

@Article{ChuG92,
  author     = {Chung, F. R. K. and Graham, R. L.},
  title      = {Quasi-random subsets of {$Z_n$}},
  year       = {1992},
  volume     = {61},
  number     = {1},
  pages      = {64--86},
  issn       = {0097-3165},
  doi        = {10.1016/0097-3165(92)90053-W},
  journal   = {Journal of Combinatorial Theory, Series A},
  mrclass    = {05C80 (05D05)},
  mrnumber   = {1178385},
  mrreviewer = {A. G. Thomason},
}

@Article{ChuG91,
  author     = {Chung, F. R. K. and Graham, R. L.},
  title      = {Quasi-random tournaments},
  journal    = {Journal of Graph Theory},
  year       = {1991},
  volume     = {15},
  number     = {2},
  pages      = {173--198},
  issn       = {0364-9024},
  doi        = {10.1002/jgt.3190150206},
  mrclass    = {05C50 (05C20)},
  mrnumber   = {1106530},
  mrreviewer = {A. G. Thomason},
}

@Article{NagRS06,
  author     = {Nagle, B. and R\"odl, V. and Schacht, M.},
  title      = {The counting lemma for regular k-uniform hypergraphs},
  journal = {Random Structures \& Algorithms},
  year       = {2006},
  volume     = {28},
  pages      = {113--179},
  }

@Article{RodS04,
  author     = {R\"odl, V. and Skokan, J.},
  title      = {Regularity lemma fork-uniform hypergraphs},
  journal = {Random Structures \& Algorithms},
  year       = {2004},
  volume     = {25},
  pages      = {1--42},
  }

@Article{ChuG91s,
  author     = {Chung, F. R. K. and Graham, R. L.},
  title      = {Quasi-random set systems},
  year       = {1991},
  volume     = {4},
  number     = {1},
  pages      = {151--196},
  issn       = {0894-0347},
  doi        = {10.2307/2939258},
  journal   = {Journal of the American Mathematical Society},
  mrclass    = {05C65 (03C13 05C80)},
  mrnumber   = {1077279},
  mrreviewer = {A. G. Thomason},
}

@Article{ChuG90,
  author     = {Chung, F. R. K. and Graham, R. L.},
  title      = {Quasi-random hypergraphs},
  journal = {Random Structures \& Algorithms},
  year       = {1990},
  volume     = {1},
  number     = {1},
  pages      = {105--124},
  issn       = {1042-9832},
  doi        = {10.1002/rsa.3240010108},
}

@Article{HavT89,
  author     = {Haviland, Julie and Thomason, Andrew},
  title      = {Pseudo-random hypergraphs},
  year       = {1989},
  volume     = {75},
  number     = {1-3},
  pages      = {255--278},
  issn       = {0012-365X},
  note       = {Graph theory and combinatorics (Cambridge, 1988)},
  doi        = {10.1016/0012-365X(89)90093-9},
  journal    = {Discrete Mathematics},
  mrclass    = {05C80 (05C65 60C05)},
  mrnumber   = {1001401},
  mrreviewer = {Zbigniew Palka},
}

@Article{CorR17,
  author     = {Coregliano, Leonardo Nagami and Razborov, Alexander A.},
  title      = {On the density of transitive tournaments},
  journal    = {Journal of Graph Theory},
  year       = {2017},
  volume     = {85},
  number     = {1},
  pages      = {12--21},
  issn       = {0364-9024},
  doi        = {10.1002/jgt.22044},
  mrclass    = {05C35 (05C20 60C99)},
  mrnumber   = {3634471},
  mrreviewer = {Jonathan A. Noel},
}

@article{Sid93,
  author    = {Alexander Sidorenko},
  title     = {A correlation inequality for bipartite graphs},
  journal   = {Graphs and Combinatorics},
  volume    = {9},
  number    = {2-4},
  pages     = {201--204},
  year      = {1993}
}

@article {ConFS10,
    AUTHOR = {Conlon, David and Fox, Jacob and Sudakov, Benny},
     TITLE = {An approximate version of {S}idorenko's conjecture},
   JOURNAL = {Geometric and Functional Analysis},
    VOLUME = {20},
      YEAR = {2010},
    NUMBER = {6},
     PAGES = {1354--1366},
      ISSN = {1016-443X}
}

@article{ConKLL18,
  title={Some advances on {S}idorenko's conjecture},
  author={Conlon, David and Kim, Jeong Han and Lee, Choongbum and Lee, Joonkyung},
  journal={Journal of the London Mathematical Society},
  volume={98},
  number={3},
  pages={593--608},
  year={2018},
  publisher={Wiley Online Library}
}

@article{ConKLL18b,
  title={Sidorenko's conjecture for higher tree decompositions},
  author={Conlon, David and Kim, Jeong Han and Lee, Choongbum and Lee, Joonkyung},
  journal={preprint arXiv:1805.02238},
  year={2018}
}

@article{ConL17,
  title={Finite reflection groups and graph norms},
  author={Conlon, David and Lee, Joonkyung},
  journal={Advances in Mathematics},
  volume={315},
  pages={130--165},
  year={2017},
  publisher={Elsevier}
}

@article {ConL21,
    AUTHOR = {Conlon, David and Lee, Joonkyung},
     TITLE = {Sidorenko's conjecture for blow-ups},
   JOURNAL = {Discrete Analysis},
      YEAR = {2021},
     PAGES = {paper no. 2, 13pp},
}

@phdthesis{Fit18,
  title={Applications of entropy to extremal problems},
  author={Fitch, Matthew},
  year={2018},
  school={University of Warwick}
}

@article{GrzLLV22,
  title={On tripartite common graphs},
  author={Grzesik, Andrzej and Lee, Joonkyung and Lidick{\'y}, Bernard and Volec, Jan},
  journal={Combinatorics, Probability and Computing},
  volume={31},
  number={5},
  pages={907--923},
  year={2022},
  publisher={Cambridge University Press}
}

@article{Lee21,
  title={On some graph densities in locally dense graphs},
  author={Lee, Joonkyung},
  journal={Random Structures \& Algorithms},
  volume={58},
  number={2},
  pages={322--344},
  year={2021},
  publisher={Wiley Online Library}
}

@misc{Par14,
  author={Parczyk, Olaf},
  title={On {S}idorenko’s conjecture},
  year={2014},
  howpublished={Master’s thesis, Freie Universit{\"a}t, Berlin}
}

@article{Sze14,
  title={An information theoretic approach to {S}idorenko's conjecture},
  author={Szegedy, Balazs},
  journal={preprint arXiv:1406.6738},
  year={2014}
}

@article{LiS11,
  title={On the logarithimic calculus and {S}idorenko's conjecture},
  author = {Li, J.~L.~X. and Szegedy, Balazs},
  journal={preprint arXiv:1107.1153},
  year={2011}
}

@article {SkoT04,
    AUTHOR = {Skokan, Jozef and Thoma, Lubos},
     TITLE = {Bipartite subgraphs and quasi-randomness},
   JOURNAL = {Graphs and Combinatorics},
    VOLUME = {20},
      YEAR = {2004},
    NUMBER = {2},
     PAGES = {255--262},
      ISSN = {0911-0119}
}

@article {BlaR65,
    AUTHOR = {Blakley, G. R. and Roy, Prabir},
     TITLE = {A {H}\"older type inequality for symmetric matrices with
              nonnegative entries},
   JOURNAL = {Proceedings of the American Mathematical Society},
    VOLUME = {16},
      YEAR = {1965},
     PAGES = {1244--1245},
      ISSN = {0002-9939}
}

@article{BucLSS21,
  title={Tournament quasirandomness from local counting},
  author={Buci{\'c}, Matija and Long, Eoin and Shapira, Asaf and Sudakov, Benny},
  journal={Combinatorica},
  volume={41},
  number={2},
  pages={175--208},
  year={2021},
  publisher={Springer}
}

@article{CorPS19,
author={Coregliano, L. N. and Parente, R. F. and Sato, C. M.},
title={On the maximum density of fixed strongly connected subtournaments},
journal={Electronic Journal of Combinatorics},
volume={26},
year={2019},
pages={P1.44}
}

@book {Lov93,
    AUTHOR = {Lov\'{a}sz, L.},
     TITLE = {Combinatorial Problems and Exercises},
 PUBLISHER = {North-Holland Publishing Co., 2nd edition},
      YEAR = {1993}
}

@article{ChaKNPSV20,
          volume = {57},
          number = {4},
          author = {T. Chan and Daniel {Kr\'al'} and Jonathan A. Noel and Yanitsa Pehova and Maryam Sharifzadeh and Jan Volec},
           title = {Characterization of quasirandom permutations by a pattern sum},
       publisher = {John Wiley \& Sons, Inc.},
            year = {2020},
         journal = {Random Structures \& Algorithms},
           pages = {920--939}
}

@article{HanKKMPSV23,
  title={No additional tournaments are quasirandom-forcing},
  author={Hancock, Robert and Kabela, Adam and {Kr\'al'}, Daniel and Martins, Ta{\'\i}sa and Parente, Roberto and Skerman, Fiona and Volec, Jan},
  journal={European Journal of Combinatorics},
  volume={108},
  pages={103632},
  year={2023},
  publisher={Elsevier}
}

@article{NoeRS25,
  title={Forcing quasirandomness in a regular tournament},
  author={Noel, Jonathan A and Ranganathan, Arjun and Simbaqueba, Lina M},
  journal={preprint arXiv:2501.11675},
  year={2025}
}

@article{ChaY24,
  title={Kruskal--{K}atona-type problems via the entropy method},
  author={Chao, Ting-Wei and Yu, Hung-Hsun Hans},
  journal={Journal of Combinatorial Theory, Series B},
  volume={169},
  pages={480--506},
  year={2024},
  publisher={Elsevier}
}

@article{ChaY24b,
  title={A Purely Entropic Approach to the Rainbow Triangle Problem},
  author={Chao, Ting-Wei and Yu, Hung-Hsun Hans},
  journal={preprint arXiv:2407.14084},
  year={2024}
}

@article{ChaY24c,
  title={When entropy meets {T}ur{\'a}n: new proofs and hypergraph {T}ur{\'a}n results},
  author={Chao, Ting-Wei and Yu, Hung-Hsun Hans},
  journal={preprint arXiv:2412.08075},
  year={2024}
}

@article{ChaALY25,
  title={Edge inducibility via local directed graphs},
  author={Chao, Ting-Wei and Antonir, Asaf Cohen and Li, Anqi and Yu, Hung-Hsun Hans},
  journal={preprint arXiv:2509.24064},
  year={2025}
}

@article{KraLN24,
  title={Forcing quasirandomness with 4-point permutations},
  author={Kr{\'a}l', Daniel and Lee, J. and Noel, Jonathan A},
  journal={preprint arXiv:2407.06869},
  year={2024}
}

@article{Gal14,
      title={Three tutorial lectures on entropy and counting}, 
      author={David Galvin},
      year={2014},
      journal={preprint arXiv:1406.7872},
}

@misc{Gow15,
  title = {Entropy and {S}idorenko’s conjecture---after {S}zegedy},
  author = {Gowers, W. T.},  
  note = {{G}owers's {W}eblog},
  year={2015},
}
\end{document}